\documentclass{amsart}

\usepackage{amssymb, amsmath}
\usepackage[all,2cell,dvips]{xy} \UseAllTwocells \SilentMatrices

\newtheorem{theorem}{Theorem}[section]
\newtheorem{proposition}[theorem]{Proposition}
\newtheorem{definition}[theorem]{Definition}
\newtheorem{question}{Question}
\newtheorem{lemma}[theorem]{Lemma}

\newtheorem{corollary}[theorem]{Corollary}

\theoremstyle{remark}
\newtheorem{example}[theorem]{Example}
\newtheorem{remark}[theorem]{Remark}

\newcommand\End{\mathop{\rm End} \nolimits}
\newcommand\NS{\mathop{\rm NS} \nolimits}
\newcommand\Pic{\mathop{\rm Pic} \nolimits}
\newcommand\Div{\mathop{\rm Div} \nolimits}
\newcommand\Br{\mathop{\rm Br} \nolimits}
\newcommand\Gal{\mathop{\rm Gal} \nolimits}

\newcommand\rk{\mathop{\rm rk} \nolimits}
\newcommand\im{\mathop{\rm im} \nolimits}

\newcommand\disc{\mathop{\rm disc} \nolimits}
\newcommand\Jac{\mathop{\rm Jac} \nolimits}

\newcommand\semi{\rtimes}
\newcommand\Q{\mathbb{Q}}
\newcommand\dC{\breve{C}}

\newcommand\dRt{\breve{R}}

\newcommand\C{\mathbb{C}}
\newcommand\Z{\mathbb{Z}}
\newcommand\F{\mathbb{F}}
\newcommand\A{\mathbb{A}}
\newcommand\p{\mathfrak{p}}
\renewcommand\P{\mathbb{P}}
\newcommand\dP{\breve{\mathbb{P}}}

\renewcommand\O{\mathcal{O}}
\newcommand\kbar{\overline{k}}

\newcommand\Ybar{\overline{Y}}

\newcommand\isom{\cong}
\newcommand\lineclass{H}

\newcommand\ra{\rightarrow}
\newcommand\dual[1]{\check{#1}}

\newcommand\checkthis[1]{}

\long\def\ignore#1{}

\begin{document}

\title[Cubic points and the Brauer-Manin obstruction on K3 surfaces]
      {Cubic points on cubic curves and the Brauer-Manin obstruction
on K3 surfaces}
 \author{Ronald van Luijk}
 \address{Dept. of Math., Simon Fraser University, Burnaby, BC, Canada, V5A 1S6}
  \email{rmluijk@gmail.com}
 \thanks{}

 \copyrightinfo{2007}{American Mathematical Society}

\begin{abstract}
We show that if over some number field there exists a
certain diagonal plane cubic curve that is locally solvable everywhere, 
but that does not have
points over any cubic galois extension of the number field, then the
algebraic part of the Brauer-Manin obstruction is not the only obstruction
to the Hasse principle for K3 surfaces.
\end{abstract}

\maketitle

\section{Introduction}

If we want to prove that a variety $X$ defined over the field
of rational numbers $\Q$ does not have any rational points, it
suffices to show that $X$ has no real points or no $p$-adic
points for some prime number $p$. For some varieties the converse
holds as well. Conics, for instance, have a rational point if and
only if they have a real point and a $p$-adic point for every
prime $p$, i.e., if and only if they have a point locally
everywhere. This is phrased by saying that conics satisfy the Hasse
principle. Selmer's famous example of the plane
curve given by
\begin{equation}\label{selmer}
3x^3+4y^3+5z^3 = 0
\end{equation}
shows that cubic curves in general do not satisfy the Hasse
principle, as this smooth curve has points everywhere locally, but
it has no points over $\Q$ (see \cite{selmerart}).

Some varieties with points locally everywhere have no global points
because of the socalled Brauer-Manin obstruction to the Hasse principle.
This obstruction is based on the Brauer group of the variety.
We refer to \cite{skor}, section 5.2,
for a good description of this obstruction.
The main idea is the following. For a smooth, projective,
geometrically integral variety $X$
over a number field $k$, the set $X(\A_k)$ of adelic points on $X$
equals the product $\prod_v X(k_v)$, where $v$ runs over all places
of $k$ and $k_v$ denotes the completion of $k$ at $v$. This product is
nonempty if and only if $X$ has points locally everywhere. Based on
class field theory, one associates to each element $g$ in the Brauer
group $\Br X$ of $X$ a certain subset $X(\A_k)^g$ of $X(A_k)$ that
contains the set $X(k)$ of $k$-points on $X$, embedded diagonally in
$X(\A_k)$. We say there is a Brauer-Manin obstruction to the Hasse
principle if $X(\A_k)\neq \emptyset$, but for some subset $B \subset
\Br(X)$ we have $\bigcap_{g \in B} X(\A_k)^g = \emptyset$, and thus
$X(k) = \emptyset$. Often one focuses on the set $B = \Br_1 X$ of
socalled algebraic elements, defined as the kernel of the map
$\Br X \rightarrow \Br X_{\kbar}$, where $\kbar$ denotes an
algebraic closure of $k$. These algebraic elements are relatively
easy to get our hands on.

For certain classes of varieties
the Brauer-Manin obstruction is the only obstruction to the Hasse
principle. For a cubic curve $C$ with finite Tate-Shafarevich group,
for instance, it is indeed true that if $C$ has points locally everywhere
and there is no Brauer-Manin obstruction, then $C$ contains a rational
point (see \cite{maninn}, or \cite{skor}, Thm. 6.2.3).
It is conjectured that the
Brauer-Manin obstruction is the only obstruction to the Hasse
principle on all curves of genus at least $2$ and all
smooth, proper, geometrically integral, rationally connected
varieties over number fields (see \cite{poonen} and \cite{CTratcon}, p. 3).

For K3 surfaces, however, it is not at all clear whether
the Brauer-Manin obstruction is the only
one. Even if in general this is not the case, it could still be true
for special K3 surfaces, such as singular K3 surfaces, which are
those with maximal Picard number $20$. A priori, it could be true that
the algebraic part already gives an obstruction, if there is any.

In this paper, a $k$-cubic point is a point defined over some
galois $(\Z/3\Z)$-extension of $k$. Our main theorem states the following.

\begin{theorem}\label{main}
Let $k$ be a number field. Suppose we have a smooth curve
$C\subset\P_k^2$ given by $ax^3+by^3+cz^3=0$, such that
\begin{enumerate}
\item the product $abc$ is not a cube in $k$,
\item the curve $C$ has points locally everywhere,
\item the curve $C$ has no $k$-cubic points.
\end{enumerate}
Then there exists a quotient of $C\times C$, defined over $k$, such that
its minimal desingularization $Y$ is a singular K3 surface
satisfying $Y(\A_k)^{\Br_1 Y} \neq \emptyset$ and $Y(k)= \emptyset$.
\end{theorem}

In other words, if a certain cubic curve exists, then the Brauer-Manin
obstruction coming from the algebraic part of the Brauer group is not
the only obstruction to the Hasse principle for singular K3 surfaces,
let alone for K3 surfaces in general.
The existence of any plane cubic curve satisfying the third condition
in Theorem \ref{main} is unknown and an interesting object of study by itself.
The curve given by (\ref{selmer}) satisfies the first two conditions
mentioned in Theorem \ref{main}. Several people have wondered whether
it satisfies the third condition as well. It turns out that this is
not the case, as the intersection points of that curve with the lines
\begin{align*}
711x+172y + 785 z &= 0,\\
657x + 124y +815z &= 0,\\
4329x+3988y + 2495 z &= 0
\end{align*}
are all $\Q$-cubic points.

In the following section we will construct the quotient $X$ of the
surface $C\times C$ associated to a general plane cubic $C$ that
will be used in the proof of Theorem \ref{main}.
We give explicit equations in the case of diagonal
cubics and cubics in Weierstrass form and discuss some of the arithmetic
properties of $X$ and its desingularization $Y$.
In section \ref{NSlattice} we will go deeper into the geometry in the case
that $k$ has characteristic $0$ and investigate the N\'eron-Severi
group of $Y$.
In section \ref{proofmain} we will prove the main theorem.
The fact that there is no Brauer-Manin obstruction will follow from
the fact that $\Br_1 Y$ is isomorphic to $\Br k$, which never yields
an obstruction.
Some related open problems are stated in the last section.

The author would like to thank Hershy Kisilevsky, Michael Stoll,
Bjorn Poonen, and Bas Edixhoven for helpful discussions and
suggestions. He also thanks Universidad de los Andes, PIMS,
University of British Columbia, and Simon Fraser University
for their support.

\section{A K3 surface associated to a plane cubic curve}

Let $k$ be any field
of characteristic not equal to $2$ or $3$, and $C$ any smooth projective
cubic curve in $\P^2_k$. We extend the regular notion of collinearity
by saying that any three points $P,Q$, and $R$ on
$C$ are collinear if the divisor $(P)+(Q)+(R)$ is linearly equivalent
with a line section of $C$. By Bezout's
theorem we know that for any two points $P$ and $Q$ on $C$ there is a
unique third point $R$ on $C$ such that $P,Q$, and $R$ are collinear.
If $P$ equals $Q$, then $R$ is the ``third'' intersection point of $C$
with the tangent to $C$ at $P$. This yields a natural isomorphism
\begin{equation}\label{CcrossC}
C \times C \,\, \isom \,\, \big\{\,(P,Q,R) \in C^3 \,\, : \,\,
\mbox{$P$, $Q$, and $R$ are collinear}\,\big\}.
\end{equation}
Let $\rho$ be the automorphism of
$C \times C$ that sends $(P,Q)$ to $(Q,R)$, where $P,Q$, and $R$ are
collinear. Under the identification of (\ref{CcrossC}) this corresponds
to sending $(P,Q,R)$ to $(Q,R,P)$. Clearly $\rho$ has order $3$.
We let $X_C$ denote the quotient $(C\times C)/\rho$, and write $X
= X_C$ if $C$ is understood. Let $\pi\colon C\times C \rightarrow X$
denote the quotient map. The surface $X_C$ is the quotient
mentioned in Theorem \ref{main}. It is also used in
\cite{hershy}, where the number of rational points on $X_C$ is
related to random matrix theory. \checkthis{ref. true?}

The fixed points of $\rho$ are exactly the nine points $(P,P)$
where $P$ runs through the flexes of $C$. Let $P$ be
such a flex and let $r$ and
$s$ be two copies of a uniformizer at $P$. Then modulo the
square of the maximal ideal at $(P,P)$ in $C\times C$, the
automorphism $\rho$ is
given by $(r,s) \mapsto (s,t)$ with $t=-r-s$, c.f.
\cite{silv}, p. 115. The subring of $k[r,s]$ of
invariants under the automorphism $(r,s)\mapsto(s,t)$ is 
generated as $k$-algebra by
$a=-rs-rt-st=r^2+rs+s^2$, $b=3rst=-3rs(r+s)$, and
$c=r^2s+s^2t+t^2r = r^3+3r^2s-s^3$. They satisfy
the equation $a^3=b^2+bc+c^2$, which locally describes the
corresponding singularity on $X_C$ up to higher degree terms, which
do not change the type of singularity. We conclude that
$X_C$ has nine singular points, all double points. Each
is resolved after one blow-up, with two smooth $(-2)$-curves
above it, intersecting each other in one point.
Let $Y_C$ denote the blow-up of $X_C$ in its singular
points. Again, we write $Y=Y_C$ if $C$ is understood.
As the singular locus of $X$ is defined over $k$, so is $Y$.

\begin{definition}
A K3 surface over $k$ is a smooth, projective, geometrically integral
surface $Z$ over $k$ with trivial canonical sheaf, for which
$H^1(Z,\O_Z)=0$.
\end{definition}

\begin{proposition}\label{YKthree}
The surface $Y$ is a smooth K3 surface.
\end{proposition}
\begin{proof}
We have seen that $\rho$ only has isolated fixed points.
The corresponding singularities are $A_2$-singularities, which
are rational double points. From the symmetry of the right-hand side
of (\ref{CcrossC}), it follows that $\rho$ fixes the
unique (up to scaling) nonvanishing regular differential
of $C\times C$. By \cite{katsura}, Thm. 2.4, these conditions imply
that a relatively minimal model of $X$ is a K3 surface. By \cite{katsura},
Lemma 2.7, this model is isomorphic to the minimal resolution $Y$ of $X$.
\end{proof}

We now discuss some of the arithmetic properties of $X$ and $Y$.

\begin{lemma}\label{XtoY}
The surface $X$ has a $k$-rational point if and only if $Y$ does.
\end{lemma}
\begin{proof}
Any $k$-rational point of $Y$ maps to a $k$-rational point of $X$.
Conversely, suppose $X$ has a $k$-rational point $P$. If $P$ is not a
singular point, then the unique point of $Y$ above $P$ is also
$k$-rational. If $P$ is a singular point, then the unique intersection
of the two irreducible components in the exceptional divisor above $P$
is a $k$-rational point on $Y$.
\end{proof}

\begin{corollary}\label{YLSE}
If $k$ is a number field and $C$ has points locally everywhere,
then so does $Y$.
\end{corollary}
\begin{proof}
Let $v$ be any place of $k$ and $k_v$ the corresponding completion.
By assumption, $C$ contains a $k_v$-rational point $P$. Then
$\pi\big( (P,P)\big)$ is a $k_v$-rational point on $X$. Applying
Lemma \ref{XtoY} to $k_v$, we find that $Y$ has a $k_v$-rational point
as well, so $Y$ has points locally everywhere.
\end{proof}

\begin{lemma}\label{Xratpoints}
The $k$-rational points of $X$ correspond to triples $(P,Q,R)$, up to
cyclic permutation, of collinear points on $C$ that are defined
over some galois $(\Z/3\Z)$-extension $l$ of $k$ and permuted by
$\Gal(l/k)$.
\end{lemma}
\begin{proof}
Suppose $l$ is a galois $(\Z/3\Z)$-extension of $k$ and $(P,Q,R)$
a triple of $l$-rational collinear points, permuted by $\Gal(l/k)$.
Then all permutations of
$(P,Q,R)$ induced by $\Gal(l/k)$ are even, so the orbit
$\{(P,Q),(Q,R),(R,P)\}$ of $\rho$ is galois invariant and
yields a $k$-rational point of $X$. Conversely, any $k$-rational
point of $X$ corresponds to a galois invariant
orbit of $\rho$, say $\{(P,Q),(Q,R),(R,P)\}$. This implies that galois
permutes $\{P,Q,R\}$, but only by even permutations, so $P, Q$, and $R$
are defined over $k$ or over some galois $(\Z/3\Z)$-extension of $k$.
They are collinear because $(P,Q)$ and $(Q,R)$ are in the same orbit of $\rho$.
\end{proof}

\begin{remark}
Note that the words ``permuted by $\Gal(l/k)$'' mean that the points
$P,Q$, and $R$ are either all defined over $k$, or they are all
conjugates.
\end{remark}

\begin{lemma}\label{Yratpoints}
The surface $Y$ has a $k$-rational point if and only if there exists a
galois $(\Z/3\Z)$-extension $l$ of $k$ such that $C$ contains three
collinear points defined over $l$ and permuted by $\Gal(l/k)$.
\end{lemma}
\begin{proof}
This follows immediately from Lemmas \ref{XtoY} and \ref{Xratpoints}.
\end{proof}

\begin{corollary}\label{Ynoratpoints}
If $C$ has no $k$-cubic points, then $Y$ has no $k$-rational points.
\end{corollary}
\begin{proof}
This follows immediately from Lemma \ref{Yratpoints}.
\end{proof}

Let $J=\Jac C$ denote the Jacobian of $C$. The following proposition is not
needed for the proof of the main theorem, but it is an interesting
fact, conveyed and proved to the author by Bjorn Poonen.

\begin{proposition}\label{bjorn}
Suppose that $J(k)$ is finite and that $3$ does not divide the order of
$J(k)$. Then the converse of Corollary \ref{Ynoratpoints} holds as well.
\end{proposition}
\begin{proof}
Suppose that $C$ contains a $k$-cubic point $P$. If $P$ is $k$-rational,
then $(P,P) \in C\times C$ maps to a $k$-rational point on $X$, so $Y$
has a $k$-rational point by Lemma \ref{XtoY}. If $P$ is not $k$-rational
then it is defined over a $(\Z/3\Z)$-extension $l$ of $k$ and has two
conjugates $Q$ and $R$. Let $L$ denote a line section of $C$. Then we have
$(P)+(Q)+(R)-L \in J(k)$, while by assumption $J(k) = 3J(k)$,
so there is an element $D \in J(k)$ such that $(P)+(Q)+(R)-L \sim 3D$.
By Riemann-Roch there exist unique points $P',Q',R'\in C(l)$ that
are linearly equivalent with $P-D, Q-D$, and $R-D$ respectively. Then
$(P')+(Q')+(R') \sim L$, so $P', Q'$, and $R'$ are collinear. Since
$\Gal(l/k)$ fixes $D$, it permutes $P', Q'$, and $R'$. By Lemma
\ref{Yratpoints} the surface $Y$ has a $k$-rational point.
\end{proof}

\begin{remark}
The Jacobian of the Selmer curve $C$ given by (\ref{selmer})
has trivial Mordell-Weil group over $\Q$. Since $C$ does not have any
$\Q$-rational points, by the proofs of Lemmas \ref{XtoY}, \ref{Xratpoints},
and Proposition \ref{bjorn}
this implies that the galois conjugates of any $\Q$-cubic point on
$C$ are collinear, so it is no surprise that the $\Q$-cubic points
mentioned in the introduction come from intersecting $C$ with a
line.
\end{remark}

We now show how to find explicit equations for $X_C$.
Let $\dP^2$ denote the dual of $\P^2$ and let
$\tau \colon C\times C \ra \dP^2$ be the map that sends $(P,Q)$ to
the line through $P$ and $Q$, and $(P,P)$ to the tangent to $C$ at
$P$. By Bezout's theorem, for a general line $L$ in $\P^2$ we have 
$\#(L \cap C) = 3$, so there are $6$ ordered pairs $(P,Q) \in C\times C$
that map under $\tau$ to $L$. We conclude that 
$\tau$ is generically $6$-to-$1$. The map $\tau$ factors
through $\pi$, inducing a $2$-to-$1$ map $\varphi \colon X \ra
\dP^2$ that is ramified over the dual $\dC$ of $C$.

$$
\xymatrix{
C \times C \ar[rr]^{\pi} \ar@<-0.2mm>@/_4mm/[rrrr]_{\tau}
&& X \ar[rr]^{\varphi} && \dP^2
}
$$

The dual $\dC$ has nine cusps, corresponding to the nine flexes of $C$.
As a double cover of $\P^2$ ramified over a curve with a cusp yields
the same singularity as the $A_2$-singularities of $X_C$, we conclude
that $X_C$ is not only birational, but in fact isomorphic to a
double cover of $\P^2$, ramified over a sextic with nine cusps.

\begin{remark}\label{functionfields}
Suppose that an affine piece of $C$ is given by $f(x,y)=0$. Then
the function field $k(C\times C)$ of $C\times C$ is the quotient field
of the ring
$$
k[x_1,y_1,x_2,y_2]/\big( f(x_1,y_1), f(x_2,y_2) \big).
$$
The line $L$ through the generic points $(x_1,y_1)$ and $(x_2,y_2)$ on
$C$ is given by
\begin{equation}\label{eqL}
L \, : \,\, y = \frac{y_2-y_1}{x_2-x_1} x +\frac{x_2y_1-x_1y_2}{x_2-x_1}.
\end{equation}
Let the coordinates of $\dP^2$ be given by $r,s,t$, where the point
$[r:s:t]$ corresponds to the line in $\P^2$ given by $rx+sy+tz=0$, or
in affine coordinates $y = -\frac{r}{s} x - \frac{t}{s}$.
Comparing this to equation (\ref{eqL}), we find that the inclusion of
function fields
$$
\tau^* \colon \, k(\dP^2) = k\left(\frac{r}{s},\frac{t}{s}\right)
\ra k(C \times C)
$$
is given by
$$
\frac{r}{s} = -\frac{y_2-y_1}{x_2-x_1},\qquad \mbox{and} \qquad
\frac{t}{s} = -\frac{x_2y_1-x_1y_2}{x_2-x_1}.
$$
Let $x_3\in k(C\times C)$ be the $x$-coordinate of the third
intersection point of the line $L$ and $C$. Then the element
$d = (x_1-x_2)(x_2-x_3)(x_3-x_1)$ is invariant under $\rho$, so it
is contained in the function field $k(X)$ of $X$, embedded in
$k(C\times C)$ by $\pi^*$. The element $d$ is not
contained in $k(\dP^2)$, as it is only invariant under even permutations
of the $x_i$. Hence $d$ generates the quadratic extension
$k(X)/k(\dP^2)$. We may therefore identify $k(X)$ with
the field $k\big(\frac{r}{s},\frac{t}{s},d\big)\subset k(C\times C)$.
\end{remark}

\begin{example}\label{exwei}
After applying a linear transformation defined over some finite extension
of $k$, we may assume that one of the $9$ inflection points
of $C$ is equal to $[0:1:0]$ and that the tangent at that point
is the line at infinity given by $z=0$. Assume this can be done
over $k$ itself. Then the
affine part given by $z=1$ is given by a Weierstrass equation and as
the characteristic of $k$ is not equal to $2$ or $3$, we can
arrange for $C$ to be given by
$$
y^2 = x^3 + Ax + B
$$
with $A,B \in k$.
With the point $\O=[0:1:0]$ as origin, $C$ obtains the structure
of an elliptic curve. Note that the inflection points of $C$ are
exactly the $3$-torsion points on the elliptic curve.

To find an explicit model for $X_C$, we find a relation
among $\frac{r}{s},\frac{t}{s}$, and $d$.
The $x$-coordinates $x_1,x_2,x_3$
of the intersection points of $C$ with the generic line $L$
given by $rx+sy+t=0$ are the solutions to the equation
$$
-\left(-\frac{r}{s}x-\frac{t}{s}\right)^2 + x^3 + Ax + B = 0.
$$
The square of $d = (x_1-x_2)(x_2-x_3)(x_3-x_1)$ is exactly the
discriminant of this polynomial, which is easy to compute.
We find that $X_C$ can be given
in weighted projective space $\P(1,1,1,3)$ with coordinates $r,s,t,u$ by
\begin{align*}
u^2 = & \, 4Br^6 - 4Ar^5t + A^2r^4s^2 + 36Br^3s^2t -4r^3t^3-18ABr^2s^4 \\
 &- 30Ar^2s^2t^2 + 24A^2rs^4t - (4A^3+27B^2)s^6    + 54Bs^4t^2 - 27s^2t^4,
\end{align*}
with $u = s^3d$. This is exactly the same as in \cite{hershy},
with $b = -r/s$ and $a= -t/s$ for their variables $a$ and $b$.
The map $\varphi\colon X \ra \dP^2$ ramifies where the right-hand side
of the equation vanishes, which describes the dual $\dC$ of $C$.
We can describe the cusps of
$\dC$ very explicitly. The cusp corresponding to $\O$ is
$[0:0:1]$. The slopes $dy/dx$ at the inflection points
of $C$ that are not equal to $\O$ are the roots of the polynomial
$$
F = u^8 + 18Au^4+108Bu^2 - 27A^2.
$$
If $\alpha$ is a root of $F$, then the corresponding inflection point
is $(x,y)=(\frac{1}{3}\alpha^2, \frac{\alpha^4+3A}{6\alpha})$.
The corresponding cusp on $\dC$ is $[r:s:t] = [6\alpha^2:-6\alpha :
3A - \alpha^4]$. Note that the splitting field of $F$ is exactly
$k(C[3])$, the field of definition of all $3$-torsion.
\end{example}

\begin{example}\label{exdiag}
Suppose $C$ is given by
$$
ax^3+by^3+c =0.
$$
Then as in Example \ref{exwei} we find that in this case
$X_C \subset \P(1,1,1,3)$ can be given by
$$
3u^2 = 2abc(cr^3s^3+br^3t^3+as^3t^3)-b^2c^2r^6-a^2c^2s^6-a^2b^2t^6,
$$
with $u =\frac{1}{9}a^2s^3d = \frac{1}{9}a^2s^3 (x_1-x_2)(x_2-x_3)(x_3-x_1)$.
Each of the coordinate axes $x=0$ and $y=0$ and the line $z=0$ at infinity
intersects the curve $C$ at three flexes. Each of the corresponding nine
cusps on $\dC$ lies on one of the coordinate axes given by $rst=0$
in $\dP^2$.
\end{example}

In the next section we will investigate the geometry of $Y$ and
find its Picard group $\Pic Y$, at least in the general case that
$C$ does not admit complex multiplication. The group $\Pic Y$ is
generated by irreducible curves on $Y$, of which we will now describe
some explicitly.

Let $\Pi$ denote the set of the nine cusps of $\dC$.
We will freely identify the elements of $\Pi$ with the flexes on $C$
that they correspond to.
A priori one would not expect there to be any conics going through
six points of $\Pi$, but it turns out there are $12$. They are
described by Proposition \ref{specialquadrics}. Fix a point $\O\in \Pi$
to give $C$ the structure of an elliptic curve.
Then $\Pi$ is the group $C[3]$ of $3$-torsion and thus $\Pi$ obtains the
structure of an $\F_3$-vectorspace in which a line is any translate of
any $1$-dimensional linear subspace. The three points on such lines are
actually collinear as points on $C$. The set of these $12$ lines 
in $\Pi$ is thus in bijection with the set of triples of collinear 
points in $\Pi$, and therefore independent of the choice of $\O$.

\begin{proposition}\label{specialquadrics}
The six points on any two parallel lines in $\Pi$ lie on a conic
whose pull-back to $X$ consists of two components.
\end{proposition}
\begin{proof}
It suffices to prove this in the setting of Example \ref{exwei}.
Let $l$ and $m$ be two parallel lines in $\Pi$. After translation by an
element of $\Pi = C[3]$ we may assume that $\O$ is not contained
in $l \cup m$. Let $P$ be a point other than $\O$
on the linear subspace that $l$ and $m$ are
translates of. By Example \ref{exwei}, the point $P$ corresponds to
a root $\alpha$ of $F$, which factors as $F = (u^2-\alpha^2)f_\alpha$.
The point $-P$ corresponds to $-\alpha$ and
the points in $l \cup m$ correspond to the roots of $f_\alpha$.
These points all lie on the conic given by
\begin{equation}\label{specialeq}
27t^2-6\alpha^2 rt +(\alpha^4+18A)r^2 +(\alpha^6+21A\alpha^2+81B)s^2=0.
\end{equation}
From Example \ref{exwei} we find that the pull back of this conic
to $X$ is given by (\ref{specialeq}) and
$$
\alpha^2 u^2 = (Ar^3+9Brs^2+3rt^2-6As^2t)^2,
$$
which indeed contains two components.
\end{proof}

\begin{remark}
Consider the situation of the proof of Proposition \ref{specialquadrics}.
Over the field $k(\alpha,\sqrt{-3})$, the polynomial $f_\alpha$
splits as the product of two cubics such that the points of $l$
correspond to the roots of one of the two cubics and the points
of $m$ correspond to the roots of the other.
\end{remark}

\begin{remark}\label{specreducible}
The fact that the conics of Proposition
\ref{specialquadrics} have a
reducible pull back to $X$ also follows without explicit
equations. Set $H = \tau^* L $ for any line $L\subset \dP^2$
that does not go through any of the $P \in \Pi$.
For each $P\in \Pi$, let $\Theta_P\in \Div Y$ be the sum of the
two $(-2)$-curves above the singular point on $X$ corresponding to $P$.
As these curves intersect each other once, we find $\Theta_P^2=-2$,
while we also have $H \cdot \Theta_P = 0$ and $H^2=2$.
Any curve $\Gamma \subset \dP^2$ of degree $m$ that goes through
$P \in \Pi$ with multiplicity $a_P$ pulls back to a curve on $X$
whose strict transformation to $Y$ is linearly equivalent with
$$
D = mH - \sum_{P \in \Pi} a_P \Theta_P.
$$
We have $D^2 = 2m^2 - 2\sum a_P^2$. For a conic $\Gamma$ through six
of the points of $\Pi$ we get $D^2=-4$. Let $p_a(D)$ be the arithmetic
genus of $D$. As the canonical divisor $K_Y$ of $Y$ is trivial,
we find from the adjunction formula $2p_a(D)-2 = D\cdot(D+K_Y)$
(see \cite{hag}, Prop. V.1.5) that $p_a(D) = -1$ is negative,
which implies that $D$ is reducible.

We can use the same idea to find $(-2)$-curves on $Y$. For instance,
the strict transformation on $Y$ of any pull back to $X$ of a line through
$2$ points of $\Pi$, or of a conic through $5$ points of $\Pi$ will be
such a curve.
\end{remark}

\section{The geometry of the surface}\label{NSlattice}

In this section we investigate some geometric properties of
$X$ and $Y$. As our main theorem only concerns characteristic $0$,
we will for convenience assume that the ground field $k$ equals $\C$
throughout this section. Several of the results, however, also
hold in positive characteristic, which we will point out at times.
We start with a quick review of lattices.

A {\em lattice} is a free $\Z$-module $L$ of finite rank, endowed with a
symmetric, bilinear, nondegenerate map $L \times L
\rightarrow \Q, \,\,(x,y) \mapsto x\cdot y$, called the {\em pairing}
of the lattice.
An {\em integral lattice} is a lattice with a $\Z$-valued pairing.
A lattice $L$ is called {\em even} if $x\cdot x
\in 2\Z$ for every $x\in L$. Every even lattice is integral.
If $L$ is a lattice and $m$ a rational number, then $L(m)$ is the
lattice obtained from $L$ by scaling its pairing by a factor $m$.
A {\em sublattice} of a lattice $\Lambda$ is a submodule $L$ of $\Lambda$
such that the induced bilinear pairing on $L$ is nondegenerate.
The orthogonal complement in $\Lambda$
of a sublattice $L$ of $\Lambda$ is
$$
L^\perp = \{\lambda \in \Lambda \,\, : \,\, \lambda \cdot x = 0
  \mbox{ for all }  x \in L \}
$$
A sublattice $L$ of $\Lambda$ is {\em primitive} if $\Lambda/L$ is
torsion-free. The minimal primitive sublattice of $\Lambda$ containing
a given sublattice $L$ is $(L^\perp)^\perp = (L\otimes \Q) \cap
\Lambda$. The {\em Gram matrix} of a lattice $L$ with respect to a given basis
$x=(x_1,\ldots, x_n)$ is $I_x = (\langle x_i,x_j \rangle)_{i,j}$.
The {\em discriminant} of $L$ is defined by $\disc L = \det I_x$
for any basis $x$ of $L$. A {\em unimodular} lattice is an integral lattice
with disriminant $\pm 1$. For any sublattice $L$ of finite
index in $\Lambda$ we have $\disc L = [\Lambda:L]^2 \cdot \disc \Lambda$.
The dual lattice of a lattice $L$ is
$$
\dual{L}= \{x \in L\otimes \Q \,\, : \,\,
x\cdot \lambda \in \Z \mbox{ for all } \lambda \in L\}.
$$
If $L$ is integral, then $L$ is contained in $\dual{L}$
with finite index $[\dual{L}:L]=|\disc L|$ and the quotient $A_L = \dual{L}/L$
is called the {\em dual-quotient} of $L$. If $L$ is a primitive sublattice
of a unimodular lattice $\Lambda$, then the dual-quotients $A_L$ and
$A_{L^\perp}$ are isomorphic as groups, and we have
$|\disc L| = |\disc L^\perp|$. For more details on these dual-quotients
and the discriminant form defined on them, see \cite{nikulin}.
\newline \smallskip

If $Z$ is a smooth projective irreducible surface, let $\Pic^0 Z
\subset \Pic Z$ denote the group of divisor classes that are
algebraically equivalent to $0$. The quotient
$\NS(Z) = \Pic Z / \Pic^0 Z$ is called the N\'eron-Severi group of $Z$. The
exponential map $\C\ra \C^*,\, z \mapsto \exp(2\pi i z)$ induces a short
exact sequence $0 \ra \Z \ra \O_Z \ra \O_Z^*\ra 1$ of sheaves on the
complex analytic space $Z_h$ associated to $Z$.
The induced long exact sequence
includes $H^1(Z_h,\O_Z)\ra H^1(Z_h,\O_Z^*) \ra H^2(Z_h,\Z)$. There is an
isomorphism $H^1(Z_h,\O_Z^*) \isom \Pic Z$ and the image of the first
map $H^1(Z_h,\O_Z)\ra H^1(Z_h,\O_Z^*)$ is exactly $\Pic^0 Z$. We conclude
that there is an embedding $\NS(Z) \hookrightarrow H^2(Z_h,\Z)$.
The intersection pairing induces a pairing on
$\NS(Z)$, which coincides with the cup-product on $ H^2(Z_h,\Z)$.
By abuse of notation we will write $H^2(Z,\Z) = H^2(Z_h,\Z)$.
See \cite{hag}, App. B.5, for statements and \cite{gaga} and
\cite{BPV}, \S IV.2, for more details.

\begin{proposition}\label{K3NSPic}
If $Z$ is a K3 surface, then $\Pic^0 Z = 0$ and we have an isomorphism
$\Pic Z \isom \NS(Z)$.
\end{proposition}
\begin{proof}
By definition we have $H^1(Z,\O_Z)=0$, so from the above we find
$\Pic^0 Z = 0$. The isomorphism follows immediately.
\end{proof}

Note that by fixing a point on $C$, the surface $C\times C$ obtains the
structure of an abelian surface.

\begin{proposition}\label{unimodular}
Let $Z$ be an abelian surface (resp. K3 surface). Then
$H(Z,\Z)$ is an even lattice with discriminant $-1$ of rank
$6$ (resp. $22$) in which $\NS(Z)$ embeds as a primitive sublattice.
\end{proposition}
\begin{proof}
The lattice $H(Z,\Z)$ is even by \cite{BPV}, Lemma VIII.3.1. If $Z$
is an abelian surface, then this lattice is unimodular by
\cite{BPV}, \S V.3, and indefinite by the Hodge index theorem,
see \cite{hag}, Thm. V.1.9.
From the classification of even indefinite unimodular lattices
we find that $H(Z,\Z)$ is isomorphic with $U^3$, where $U$ is the
hyperbolic lattice with discriminant $-1$, see \cite{serre}, Thm. V.5.
\checkthis{direct ref for abelian?}
A similar argument holds for K3 surfaces, see
\cite{BPV}, Prop. VIII.3.3 (VIII.3.2 in first edition).
This implies that in both cases the map $H^2(Z,\Z)\ra H^2(Z,\C)$ is
injective, so $\NS(Z)$ is the image of $\Pic Z$ in $H^2(Z,\C)$. By
\cite{BPV}, Thm. IV.2.13 (IV.2.12 in first edition), this image is the
intersection of $H^2(X,\Z)$ with the $\C$-vectorspace
$H^{1,1}(X,\Omega_X)$ inside $H^2(X,\C)$. This implies the
last part of the claim.
\end{proof}

\begin{remark}
From Propositions \ref{K3NSPic} and \ref{unimodular} it follows that
linear, algebraic, and numerical equivalence all
coincide on a complex K3 surface. In positive characteristic this is
the case as well, see \cite{bombmum}, Thm. 5.
\end{remark}

\begin{lemma}\label{Dsquared}
Let $Z$ be an abelian variety or a K3 surface. Let $D$ be a curve on
$Z$ with arithmetic genus $p_a(D)$. Then we have $D^2 = 2p_a(D)-2$.
\end{lemma}
\begin{proof}
The canonical divisor $K_Z$ is trivial for both abelian varieties
and K3 surfaces. The adjunction formula therefore gives
$2p_a(D)-2 = D\cdot (D+K_Z)=D^2$ (see \cite{hag}, Prop. V.1.5).
\end{proof}

\begin{lemma}\label{Dis}
Take any point $R\in C$ and define the divisors
$$
D_1 = \{\,(P,Q) \,\, : \,\, P,Q,R \mbox{ collinear}\,\}, \quad
D_2 = C \times \{R\}, \quad
D_3 = \{R\} \times C
$$
on $C\times C$.  The automorphism
$\rho$ acts on the $D_i$ as the permutation $(D_1 \,\, D_2 \,\, D_3)$.
The images in $\NS(C\times C)$ of the $D_i$
are independent of the choice of $R$.
We have $D_i \cdot D_j = 1$ if $i\neq j$ and $D_i^2=0$.
The elements $D_1,D_2,D_3$ are numerically independent.
\end{lemma}
\begin{proof}
The first statement is obvious. All fibers of the projection of
$C\times C$ onto the second copy of $C$ are algebraically equivalent
to each other,
so the image of $D_2$ in $\NS(C\times C)$ is independent of $R$.
As $D_1$ and $D_3$ are in the orbit of $D_2$ under $\rho$, their
images are independent of $R$ as well. The divisors
$D_2$ and $D_3$ intersect each other transversally in one point, so
$D_2 \cdot D_3 = 1$. As $D_2$ is isomorphic to $C$, the genus
$p_a(D_2)$ of $D_2$ equals $1$, so Lemma \ref{Dsquared} gives
$D_2^2 = 0$. The other intersection numbers follow
by applying $\rho$. The last statement follows immediately.
\end{proof}

For any $R\in C$ the three divisors $D_i$ of Lemma \ref{Dis} all map
birationally to the same curve on $X_C$, namely the pull back
$\varphi^*(\dRt)$ of the dual $\dRt$ of $R$, i.e.,
the line in $\dP^2$ consisting of all lines in $\P^2$ going through $R$.
%

For each $P \in \Pi$,
let $L_P\subset \NS(Y)$ be the lattice generated by the two
$(-2)$-curves above the singularity on $X$ corresponding to $P$.
Then the lattice $L$ generated by all these $(-2)$-curves is
isomorphic to the orthogonal direct sum
$\bigoplus_{P \in \Pi} L_P$. For each $P\in \Pi$,
the dual-quotient $A_{L_P}$ is a $1$-dimensional $\F_3$-vectorspace.
Let $\Lambda = (L^\perp)^\perp$ be the minimal primitive sublattice
of $\NS(Y)$ that contains $L$. Then $\Lambda$ is contained in the dual
$\dual{L}$ of $L$, so $\Lambda/L$ is a subspace of
the dual-quotient $A_L \isom \bigoplus_{P \in \Pi} A_{L_P}$.

\begin{lemma}\label{code}
We can identify each $A_{L_P}$ with $\F_3$ and give $\Pi$ the structure of
an $\F_3$-vectorspace in such a way that
$$
\Lambda/L \subset \bigoplus_{P \in \Pi} A_{L_P} \isom \F_3^\Pi
$$
consists of all affine linear functions $\Pi \rightarrow \F_3$.
\end{lemma}
\begin{proof}
See \cite{bertin}, Thm. 2.5. The subspace $\Lambda/L$ corresponds
to $L_3$ as defined on page 269 of \cite{bertin}.
\end{proof}

\begin{corollary}\label{discLambda}
We have $[\Lambda : L] = 27$ and $\disc \Lambda = 27$.
\end{corollary}
\begin{proof}
The first equality follows from Lemma \ref{code} and the fact
that there are $27$ affine linear functions $\F_3^2 \rightarrow \F_3$.
The second equality then follows from the equation
$[\Lambda : L]^2 \cdot \disc \Lambda = \disc L = 3^9$.
\end{proof}

\begin{remark}\label{weil}
Fix a point $\O \in \Pi$. Then $C$ obtains the structure of an elliptic
curve and $\Pi$ corresponds to the group $C[3]$ of $3$-torsion elements,
which naturally has the structure of an $\F_3$-vectorspace.
The $24$ nonconstant affine linear
functions $\Pi \rightarrow \F_3$ correspond to the
irreducible components of the $12$ pull backs of the conics of
Proposition \ref{specialquadrics}.
\end{remark}

Let $T_{C\times C}$ and $T_Y$ denote the othogonal
complements of $\NS(C\times C)$ and $\NS(Y)$ in
$H^2(C\times C, \Z)$ and $H^2(Y,\Z)$ respectively.
%
%
The main result of this section is the following.

\begin{proposition}\label{transcendental}
There is a natural isomorphism $T_Y \isom T_{C\times C}(3)$ of lattices.
\end{proposition}
\begin{proof}
For convenience write $H_{C\times C}^\rho = H^2(C\times C,\Z)^{\langle\rho\rangle}$.
From Katsura's table in \cite{katsura}, p. 17, we find
$\rk H_{C\times C}^\rho = 4$ and the remaining eigenvalues of $\rho^*$ acting
on $H^2(C\times C,\Z)$ are $\zeta$ and $\zeta^2$, where
$\zeta$ is a primitive cube root of unity. Let $\Gamma' \subset
\NS(C\times C)$ denote the sublattice generated by the
$D_i$ of Proposition \ref{Dis}, and set $D = D_1+D_2+D_3$.
By Proposition \ref{Dis}, $D$ is fixed by $\rho$. As $\rho^*$
acts unitarily on $H^2(C\times C,\C)$, its eigenspaces
corresponding to different eigenvalues are orthogonal.
We conclude from Proposition \ref{Dis} that the orthogonal
complement $\Gamma$ of $\langle D \rangle$ inside $\Gamma'$
corresponds to the eigenvalues $\zeta$ and $\zeta^2$, which
in turn implies $H_{C\times C}^\rho = \Gamma^\perp$ inside the
unimodular lattice $H^2(C\times C,\Z)$, because $H_{C\times C}^\rho$
corresponds to the eigenvalue $1$. The lattice $\Gamma$ is
generated by $D_1-D_2$ and $D_2-D_3$ and has discriminant $3$,
so it is primitive and we also have 
$|\disc H_{C\times C}^\rho| = |\disc \Gamma|=3$.
Set $N = \Gamma'^\perp$. Taking orthogonal complements in
$\Gamma \subset \Gamma' \subset \NS(C \times C)$ we find
$T_{C\times C}\subset N \subset H_{C\times C}^\rho$.
From the fact that $(\Gamma'/\Gamma) \otimes \Q$ is generated by $D$,
it follows that $N$ is the orthogonal complement of $D$
inside $H_{C\times C}^\rho$.

Let $H_X$ denote the orthogonal complement of $L$ (or $\Lambda$)
in the unimodular lattice $H^2(Y,\Z)$, so that by Corollary 
\ref{discLambda} we have
$|\disc H_X| = |\disc \Lambda| = 27$. Recall that $\pi$ denotes
the quotient map $C \times C \rightarrow X_C$. There are maps
$\pi^*\colon \, H_X \rightarrow H_{C\times C}^\rho$ and $\pi_*
\colon \, H_{C\times C}^\rho \rightarrow H_X$ such that $\pi^*$
and $\pi_*$ send transcendental cycles to transcendental cycles
and $$
\begin{array}{rll}
(i) & \pi^*(x) \cdot \pi^*(y) = 3 x\cdot y, & \forall x,y \in H_X, \cr
(ii) & \pi_*(x) \cdot \pi_*(y) = 3 x\cdot y, & \forall x,y \in H_{C\times C}^\rho, \cr
(iii) & \pi_*(\pi^*(x)) = 3x, & \forall x \in H_X, \cr
(iv) & \pi^*(\pi_*(x)) = 3x, & \forall x \in H_{C\times C}^\rho, \cr
\end{array}
$$ see \cite{Inose}, \S 1, and \cite{bertin}, p. 273. From (iv)
and the fact that $H^2(C\times C,\Z)$ is torsion-free, we find
that $\pi_*$ is injective. Therefore, by (ii) we have an
isomorphism $\pi_*(H_{C\times C}^\rho) \isom H_{C\times
C}^\rho(3)$ and $$ |\disc \pi_*(H_{C\times C}^\rho)| = |\disc
H_{C\times C}^\rho(3)| = 3^{\rk H_{C\times C}^\rho} \cdot |\disc
H_{C\times C}^\rho| = 3^5. $$ From
$$
[H_X:\pi_*(H_{C\times C}^\rho)]^2 \cdot |\disc H_X| =
|\disc \pi_*(H_{C\times C}^\rho)|,
$$
we then find $[H_X:\pi_*(H_{C\times C}^\rho)]=3$.

Recall that $\pi_*(D_1)=\pi_*(D_2)=\pi_*(D_3)=\lineclass$,
where $\lineclass$ denotes the class of the strict transformation
of the pull back of a line in $\dP^2$ to $X$.
Therefore, $\pi_*(D) = 3\lineclass$. From $9\nmid 6 = D^2$
we conclude $D \not \in 3H_{C\times C}^\rho$, and thus $3\lineclass = \pi_*(D)
\not \in 3\pi_*(H_{C\times C}^\rho)$, which implies $H \in H_X \setminus
\pi_*(H_{C\times C}^\rho)$.
As the index $[H_X:\pi_*(H_{C\times C}^\rho)]=3$ is prime, it follows that
$H_X/\pi_*(H_{C\times C}^\rho)$ is generated by $\lineclass$, so the
orthogonal complement
$\pi_*(N)$ of $3\lineclass = \pi_*(D)$ in $\pi_*(H_{C\times C}^\rho)$
is primitive in $H_X$.
From the primitive inclusion $\pi_*(T_{C\times C})\subset \pi_*(N)$
it follows that also
$\pi_*(T_{C\times C})$ is primitive in $H_X$, and thus in $H^2(Y,\Z)$.

%
From $L \subset \NS(Y)$ we find $T_Y\subset H_X$.
From (i)-(iv) and the fact that $\pi_*$ and $\pi^*$ send transcendental
elements to transcendental elements, we find
$3T_Y \subset \pi_*(T_{C\times C}) \subset T_Y$.
This implies $\rk T_Y = \rk T_{C\times C}$ and
together with the fact that $\pi_*(T_{C\times C})$ is primitive in
$H^2(Y,\Z)$, it follows that we have $T_Y = \pi_*(T_{C\times C})
\isom T_{C\times C}(3)$, where the last isomorphism follows from
$\pi_*(H_{C\times C}^\rho) \isom H_{C\times C}^\rho(3)$.
\end{proof}

\begin{remark}
Proposition \ref{transcendental} is mentioned without proof in
\cite{onsiper}, where it is claimed that the proof is exactly the same
as in the classical
Kummer case, where $Y$ is the desingularization of the quotient $X$ of an
abelian surface $A$ by an involution $\iota$. Indeed there are many
similarities between the proof of the classical case
(see for instance \cite{mor}, Prop. 4.3)
and the one just presented, but there are some essential
differences. A first difference
is that in the classical case $\iota$ acts trivially on $H^2(A,\Z)$. There is,
however, a much more significant difference. In the classical case the
lattice $\pi_* H(A,\Z)^{\langle \iota \rangle}$ is easily proved to be
primitive in the orthogonal
complement $H_X$ of the lattice $L$ generated by the $16$ exceptional
divisors on $Y$. This immediately implies that $\pi_* T_A$ is primitive
in $H^2(Y,\Z)$. As in our case the index $[H_X:\pi_* H_{C\times C}^\rho]=3$
is not trivial, the classical proof does certainly not
directly apply.
\end{remark}

As before, let $\Jac C$ denote the Jacobian of $C$ and $\End \Jac(C)$ its
endomorphism ring, which is isomorphic to $\Z$ or an order in
either an imaginary quadratic field or a quaternion algebra.

\begin{proposition}\label{rkNS}
We have $\rk \NS(C\times C) = 2+ \rk \End \Jac C$.
\end{proposition}
\begin{proof}
Note that for a curve $D$ the group $\Pic^0 D$ is the kernel
of the degree map $\Pic D \ra \Z$, so that we have an isomorphism
$\NS(D)\isom \Z$. The statement then follows from
\cite{shioda}, App., or \cite{argentin}, Thm. 3.11.
\checkthis{real reference? Mumford's Ab. Var's?}
\end{proof}

\begin{proposition}\label{picnumber}
With $r = \rk \End \Jac C$ we have 
$$
\rk \NS(Y) = 18 + r \quad \mbox{and} \quad
\disc \NS(Y) = 3^{4-r} \disc \NS(C\times C).
$$
\end{proposition}
\begin{proof}
From Propositions \ref{unimodular} and \ref{rkNS} we get
$\rk T_{C\times C} = 6 - \rk \NS(C\times C) = 4-r$. From Propositions
\ref{unimodular} and \ref{transcendental} we then conclude
$$
\rk \NS(Y) = 22 - \rk T_Y = 22 - \rk T_{C\times C} = 18+r.
$$
From Proposition \ref{transcendental} we also get
$$
\disc T_Y = \disc T_{C\times C}(3) = 3^{\rk T_{C\times C}} \disc
T_{C\times C} = 3^{4-r} \disc T_{C\times C}.
$$
From Proposition \ref{unimodular} we then find
$$
\disc \NS = -\disc T_Y = -3^{4-r}\disc T_{C\times C} =
3^{4-r} \disc \NS(C \times C).
$$
\end{proof}

\begin{remark}
The conclusion $\rk \NS(Y) = 18 + \rk \End \Jac C$ of Proposition
\ref{picnumber} is much weaker than the statement of Proposition
\ref{transcendental} as it suffices to work with coefficients in
$\Q$ or $\C$ instead of $\Z$ in the cohomology.
%
%
By working with \'etale cohomology instead, we can deduce the
same equation in positive characteristic. For most of the details, see
\cite{katsura}, which only gives $\rk \NS(Y) \geq 19$
(\cite{katsura}, p. 17).
\end{remark}

\begin{corollary}
If $C$ does not admit complex multiplication, then the N\'eron-Severi
lattice $\NS(Y_C)$ has rank $19$, discriminant $54$, and is generated
by the pull back $H$ of a line in $\dP^2$,
the irreducible components above the $P\in \Pi$, and the
irreducible components of the pull backs of the conics of
Proposition \ref{specialquadrics}.
\end{corollary}
\begin{proof}
If $C$ does not admit complex multiplication, then
$\End \Jac C$ has rank $r=1$ by definition, so by Proposition
\ref{picnumber} we have $\rk \NS(Y_C)=19$ and $27 | \disc \NS(Y)$.
The lattice $\Lambda$ generated by the
irreducible components above the $P\in \Pi$ and the conics of
Proposition \ref{specialquadrics} has rank $18$ and discriminant
$27$ by Corollary \ref{discLambda} and Remark \ref{weil}. The
class $H$ is orthogonal to $\Lambda$ and satisfies $H^2=2$, so
$\langle H \rangle \oplus \Lambda$ has discriminant $27 \cdot 2 = 54$
and rank $19$, and thus finite index in $\NS(Y)$.
From
$$
\disc \NS(Y) \cdot [\NS(Y) \, : \, \langle H \rangle \oplus \Lambda]^2
= \disc \langle H \rangle \oplus \Lambda = 54,
$$
and the fact that $27|\disc \NS(Y)$, we find that the index equals $1$,
so $\NS(Y) = \langle H \rangle \oplus \Lambda$.
\end{proof}


\section{Diagonal cubics and the proof of the main theorem}\label{proofmain}

Let $C\subset \P^2$ be a diagonal cubic, defined over a number field $k$,
given by $ax^3+by^3+cz^3=0$. Let $\zeta \in \kbar$ denote a primitive
cube root of unity. Then $J=\Jac C$ is an elliptic curve of
$j$-invariant $0$, with endomorphism ring $\End J \isom \Z[\zeta]$.
Consider the automorphism $\zeta_x \colon \, [x:y:z] \mapsto
[\zeta x : y: z]$ of $C$. The line through a point $P=[x_0:y_0:z_0]$ and
$\zeta_x P$ is given by $z_0y - y_0z = 0$ and also goes through
$\zeta_x^2 P$, so the map
$\tau \colon \, C \times C \rightarrow \dP^2$ sends both
$(P,\zeta_x P)$ and $(P,\zeta_x^2 P)$ to $[0:z_0:-y_0]$. It follows
that both curves
\begin{align*}
D_4=&\left\{ \, (P, \zeta_x P ) \in C\times C \,\, : \,\, P \in C\, \right\}\cr
D_5=&\left\{ \, (P, \zeta_x^2 P ) \in C\times C \,\, : \,\, P \in C\, \right\}
\end{align*}
map under $\tau$ to the line $L_r$ in $\dP^2$ given by $r=0$. However,
the two curves are both fixed by $\rho$, so they map to different
curves in $X_C$. Hence, the pull back to $X_C$ of $L_r$ consists of two
irreducible components.
The same could be concluded from an argument similar to that in
Remark \ref{specreducible}. The line $L_r$ goes through
three cusps of $\dC$ (see Example \ref{exdiag}), so the strict
transform $D$ on $Y_C$ of the pull back to $X_C$ of $L_r$ is
linearly equivalent to $H-\sum_{P \in \Pi \cap L_r}\Theta_P$, which
implies $D^2=-4<-2$, so $D$ is reducible. Obviously,
the same holds for the lines given by $s=0$ and $t=0$.

By Proposition \ref{rkNS} we have
$\rk \NS(C\times C) = 4$. The divisors
$D_1,D_2,D_3$ of Proposition \ref{Dis} and $D_4$ from above
mutuallly intersect each other exactly once.  They are all isomorphic
to $C$, so they have genus $1$ and we have $D_i^2=0$ by Lemma
\ref{Dsquared}. It follows that the $D_i$ ($1\leq i \leq 4$)
generate a sublattice $V$ of $\NS(C\times C)$ of rank $4$ and discriminant
$-3$. As this discriminant is squarefree, we find $\NS(C\times C) = V$,
and thus $\disc \NS(C\times C) = -3$. By Proposition \ref{picnumber}
we conclude $\disc \NS(Y) = -27$ and $\rk \NS(Y)=20$.
We will now describe the N\'eron-Severi group more concretely.

By Example \ref{exdiag},
the pull back of the line $L_r$ given by $r=0$ to $X_C \subset
\P(1,1,1,3)$ is given by $r=0$ and $-3u^2 = a^2(bt^3-cs^3)^2$.
We will denote the two components by $D_r^\omega$, where
$\omega \in \{\zeta,\zeta^2\}$ is such that
$1+2\omega = a(bt^3-cs^3)/u$ on the corresponding component.
We have $D_r^\zeta+D_r^{\zeta^2} \sim H -\sum_{P\in \Pi \cap L_r} \Theta_P$ with 
$\Theta_P$ as in Remark \ref{specreducible}.
Similarly, $D_s^\omega$ and $D_t^\omega$ denote the irreducible
components above $s=0$ and $t=0$ with $\omega$ such that $1+2\omega$ equals
the value along the corresponding component of $b(cr^3-at^3)/u$
and $c(as^3-br^3)/u$ respectively.

Choose elements $\alpha, \beta \in \kbar$ such that
$\alpha^3 = -c/b$ and $\beta^3= -a/c$, and set $\gamma = -
\alpha^{-1}\beta^{-1}$, so that $\gamma^3 = -b/a$.
The flexes of $C$ are given by $[0:\alpha:\zeta^i]$,
$[\zeta^i:0:\beta]$, and $[\gamma:\zeta^i:0]$, with
$0 \leq i \leq 2$. The corresponding cusps of $\dC$ are
$[0: -\zeta^i : \alpha]$, $[\beta:0:-\zeta^i]$,
and $[-\zeta^i: \gamma:0]$ respectively. Let $\O, P$, and $Q$
denote the cusps $[0:-1:\alpha]$, $[0:-\zeta:\alpha]$, and
$[\beta:0:-1]$ respectively. Identifying the cusps of $\dC$ with
the flexes of $C$, the curve $C$ gets the structure of an elliptic
curve with origin $\O$. The following addition table shows what
the other cusps correspond to.
$$
\begin{array}{c|ccc}
   &     \O               &       Q            &    -Q \cr
\hline
\O & [0:-1:\alpha]        & [\beta:0:-1]       & [-1:\gamma:0] \cr
P  & [0:-\zeta:\alpha]    & [\beta:0:-\zeta]   & [-\zeta:\gamma:0] \cr
-P & [0:-\zeta^2:\alpha]  & [\beta:0:-\zeta^2] & [-\zeta^2:\gamma:0] \cr
\end{array}
$$
The curve in $\dP^2$ given by
$$
r^2+\alpha^2\beta^2 s^2+\beta^2t^2+\alpha\beta  rs -\beta rt -
\alpha \beta^2 st = 0
$$
is one of the conics of Proposition \ref{specialquadrics}, going through
the points $nQ\pm P$ for any integer $n$. Its pullback to $X_C \subset
\P(1,1,1,3)$ is given by the same equation together with
$\alpha^2u = \pm \beta^2c^2rst$, therefore
containing two irreducible components that we will denote according
to the sign in the equation by $D_{\alpha,\beta}^\pm$. By choosing
$\alpha$ and $\beta$ differently, we get nine out of the twelve conics
of Proposition \ref{specialquadrics}. The remaining three are the three
possible pairs out of the three lines given by $rst=0$.

Take any $\alpha' \in \{\alpha,\zeta\alpha,\zeta^2\alpha\}$ and consider
the affine coordinates $u'=u/s^3$, $r'=r/s$, and $t'=t/s+\alpha'$.
By Example \ref{exdiag}, $X_C$ is locally given by
$u'^2 = -3a^2b^2\alpha'^4t'^2 + (\mbox{\em higher order terms})$,
where the point $R = [0:-1:\alpha']$ corresponds to $(0,0,0)$.
Therefore, the square of the ratio $3ab\alpha'^2t'/u'$ equals $-3$
on both irreducible components of the exceptional divisor of
the blow-up at $R$. We denote these components
by $\Theta_{r,\alpha'}^\omega$ or $\Theta_R^\omega$,
with $\omega\in \{\zeta,\zeta^2\}$ such that
$1+2\omega = 3ab\alpha'^2t'/u' =
3ab\alpha'^2s^2(t+\alpha's)/u$ on the corresponding component.
Note that $\Theta_R^\zeta + \Theta_R^{\zeta^2} =
\Theta_R$, c.f. Remark \ref{specreducible}.
Similarly, for $\beta' \in \{\beta,\zeta\beta,\zeta^2\beta\}$
and $R = [\beta':0:-1]$, we denote the irreducible components above
$R$ by $\Theta_{s,\beta'}^\omega$ or $\Theta_R^\omega$, with
$\omega$ such that $1+2\omega = 3bc\beta'^2t^2(r+\beta't)/u$.
For $\gamma' \in \{\gamma,\zeta\gamma,\zeta^2\gamma\}$
and $R = [-1:\gamma':0]$, we denote the irreducible components above
$R$ by $\Theta_{t,\gamma'}^\omega$ or $\Theta_R^\omega$ in such a way that
we have $1+2\omega = 3ac\gamma'^2r^2(s+\gamma'r)/u$ on the corresponding
component.

We will see that the
$43$ divisors $H$, $\Theta_R^{\omega}$, $D_{v}^\omega$ and
$D_{\alpha',\beta'}^\pm$ generate the N\'eron-Severi group.
Their intersection numbers are easily computed from the above
and given by
$$
\begin{array}{lll}
H^2=2,                     & H\cdot D_v^\omega = 1,
      & D_{\alpha',\beta'}^+\cdot D_{\alpha'',\beta''}^-=0,\cr
H\cdot \Theta_R^\omega= 0, & H\cdot D_{\alpha',\beta'}^{\pm} =2,
      & D_v^\omega\cdot D_{\alpha',\beta'}^{\pm}=0,\cr
\end{array}
$$
\begin{align*}
\Theta_{v,\delta}^{\omega_1}\cdot\Theta_{w,\delta'}^{\omega_2}&=
\left\{
\begin{array}{rl}
-2 & \mbox{if $v=w$ and $\delta=\delta'$ and $\omega_1=\omega_2$,}\cr
1 & \mbox{if $v=w$ and $\delta=\delta'$ and $\omega_1 \neq \omega_2$,} \cr
0 & \mbox{otherwise,} \cr
\end{array}
\right.\cr
D_v^{\omega_1} \cdot D_w^{\omega_2} &=
\left\{
\begin{array}{rl}
-2 & \mbox{if $v=w$ and $\omega_1=\omega_2$,}\cr
1 & \mbox{if $v\neq w$ and $\omega_1 \neq \omega_2$,} \cr
0 & \mbox{otherwise,} \cr
\end{array}
\right.\cr
D_{\alpha',\beta'}^\epsilon\cdot D_{\alpha'',\beta''}^\epsilon&=
\left\{
\begin{array}{rl}
-2 & \mbox{if $\alpha''=\alpha'$ and $\beta''=\beta'$,} \cr
1 & \mbox{if $\alpha''\alpha'^{-1}= \beta''\beta'^{-1} \neq 1$,} \cr
0 & \mbox{otherwise,} \cr
\end{array}
\right.\cr
D_v^{\omega_1} \cdot \Theta_{w,\delta}^{\omega_2} &=
\left\{
\begin{array}{rl}
1 & \mbox{if $v=w$ and $\omega_1=\omega_2$,} \cr
0 & \mbox{otherwise,} \cr
\end{array}
\right.\cr
D_{\alpha',\beta'}^\epsilon \cdot \Theta_{v,\delta}^\omega &=
\left\{
\begin{array}{rl}
1 & \mbox{if $v=r$ and $\alpha'/\delta = \omega^\epsilon$}, \cr
1 & \mbox{if $v=s$ and $\beta'/\delta = \omega^\epsilon$}, \cr
1 & \mbox{if $v=t$ and $\gamma'/\delta = \omega^\epsilon \quad
                                      (\gamma' = -(\alpha'\beta')^{-1})$}, \cr
0 & \mbox{otherwise}
\end{array}
\right.
\end{align*}
for any $R \in \Pi$, $v,w \in \{r,s,t\}$, $\omega \in \{\zeta,\zeta^2\}$,
$\epsilon \in \{+,-\}$, $\alpha',\alpha'' \in
\{\alpha,\zeta\alpha,\zeta^2\alpha\}$,
$\beta',\beta'' \in \{\beta,\zeta\beta,\zeta^2\beta\}$,
$\delta\in \{\zeta^i\alpha, \zeta^i\beta, \zeta^i \gamma \,\, : \,\, 0 \leq i \leq 2\}$
and $\omega^+=\omega$ and $\omega^- = \omega^{-1}$.
The names of all the divisors are chosen to optimize symmetry.
In particular, under the identification of the multiplicative group
$\mu_3$ with the additive group $\F_3$, and with the
$\F_3$-vectorspace structure on $\Pi$ coming from the addition on
the elliptic curve, the superscripts of the $\Theta_R^\omega$
correspond exactly with the choices that need to be made in Lemma
\ref{code}. Many of these intersection numbers follow from that
lemma, but we preferred this concrete distinction between the
$\Theta_R^\zeta$ and $\Theta_R^{\zeta^2}$, which more easily reveals
the intersection numbers with the $D_v^\omega$ and where the galois
action on these divisors is given by the action on the
superscripts and subscripts.

\begin{proposition}\label{generators}
The N\'eron-Severi group of $Y$ has rank $20$ and discriminant $-27$.
It is generated by the galois-invariant set
\begin{align*}
\left\{D_r^\zeta,D_r^{\zeta^2}\right\} &\cup
\left\{\Theta_R^\omega\,\,:\,\,R\in\Pi,\,\omega\in\{\zeta,\zeta^2\}\right\}\cr
&\cup \left\{D_{\alpha', \beta'}^+ \,\,:\,\, \alpha'\in \{\alpha,\zeta\alpha,
\zeta^2\alpha\}, \beta' \in \{\beta,\zeta\beta,\zeta^2\beta\}\right\}.
\end{align*}
\end{proposition}
\begin{proof}
The $29$ given divisors generate a lattice $\Lambda$ of rank $20$ and
discriminant $-27$. As we had already seen $\rk \NS(Y)=20$ and
$\disc \NS(Y) = -27$, we conclude $\Lambda= \NS(Y)$.
\end{proof}

\begin{proposition}\label{Hone}
If $abc$ is not a cube in $k$, then we have
$H^1(k,\Pic \overline{Y}) = \{1\}$.
\end{proposition}
\begin{proof}
From Proposition \ref{generators} we know a galois-invariant set
of generators for $\Pic \overline{Y}$.
Let $\rho,\sigma,\tau$ be the automorphisms of $\Pic \overline{Y}$
induced by acting as follows on the superscript and subscripts.
\begin{align*}
\rho\colon\,   &(\alpha,\beta,\zeta) \mapsto (\zeta\alpha,\beta,\zeta), \cr
\sigma\colon\,  &(\alpha,\beta,\zeta) \mapsto (\alpha,\zeta\beta,\zeta), \cr
\tau\colon\, &(\alpha,\beta,\zeta) \mapsto (\alpha,\beta,\zeta^2).
\end{align*}
The automorphisms $\rho$ and $\sigma$ commute and the group $G=\langle
\rho, \sigma, \tau \rangle$ is isomorphic to
the semi-direct product $\langle \rho, \sigma \rangle
\semi \langle \tau \rangle \isom (\Z/3\Z)^2 \semi \Z/2\Z$, where
$\tau$ acts on $\langle \rho, \sigma \rangle$ by inversion.
The group $\Pic \overline{Y}$ is defined over $k(\zeta,\alpha,\beta)$, 
so we have $H^1(k,\Pic \Ybar) \isom H^1(k(\zeta,\alpha,\beta)/k,\Pic \Ybar)$.
The galois group $\Gal(k(\zeta,\alpha,\beta)/k)$ injects into $G$. In for
instance {\sc magma} we can compute
$H^1(H,\Pic \overline{Y})$ for every subgroup $H$ of $G$, see \cite{elec}.
It turns out that if $H^1(H,\Pic \overline{Y})$ is nontrivial, then $H$ is
contained in $\langle \rho \sigma , \tau\rangle$. This implies that if
$H^1(k(\zeta,\alpha,\beta)/k,\Pic \overline{Y})$ is nontrivial,
then $\Gal(k(\zeta,\alpha,\beta)/k)$ injects into the group 
$\langle \rho \sigma , \tau\rangle$, so $\beta/\alpha$ is fixed
by the action of galois and therefore contained
in $k$, so $abc = (c\beta/\alpha)^3$ is a cube in $k$.
\end{proof}

\begin{proof}[Proof of Theorem \ref{main}]
By Proposition \ref{YKthree} the surface $Y$ is a K3 surface. As its
Picard number equals $20$ by Proposition \ref{generators}, it is
in fact a singular K3 surface.
By Corollary \ref{Ynoratpoints} we have $Y(k) = \emptyset$.
By Corollary \ref{YLSE} the surface $Y$ has points locally everywhere,
so $Y(\A_k) \neq \emptyset$.
By the Hochschild-Serre spectral sequence we have an exact sequence
$\Br k \ra \Br_1Y \ra H^1(k,\Pic \overline{Y})$, 
see \cite{skor}, Cor. 2.3.9. By Proposition
\ref{Hone} we find $\Br_1Y = \im \Br k$. As $\Br k$ never yields
a Brauer-Manin obstruction,
we find $Y(\A_k)^{\Br_1Y} \neq \emptyset$.
\end{proof}

\begin{remark}
For some fields $k$ the conditions (2) and (3) of Theorem \ref{main}
imply condition (1). To see this, assume that we have a smooth curve
$C \subset \P_k^2$ given by $ax^3+by^3+cz^3=0$ that satisfies conditions
(2) and (3) while $abc$ is a cube in the number field $k$, say $abc=d^3$.
For any $\lambda \in k$ we consider $e=\lambda^3 b/a$.
Then the linear transformation $[x:y:z]\rightarrow
[x:\lambda^{-1} y:\lambda^{-2}b^{-1}dz]$ sends $C$ isomorphically to the
curve given by $x^3+ey^3+e^2z^3$, so without loss of generality we will assume
$a=1$ and $b=e$ and $c=e^2$. By picking $\lambda$ suitably, we may also
assume that $e$ is integral.
Let $\p$ be a place of $k$. If $3\nmid v_\p(e)$, then $C$
is not locally solvable at $\p$ as the three terms of the defining equation
have different valuations, so we find $3|v_\p(e)$ for each place $\p$ of $k$,
which means that the ideal $(e)$ is the
cube of some ideal $I$ of the ring $\O_k$ of integers of $k$. Now assume that
the class number of $k$ is not a multiple of $3$. Then from the fact that
$I^3$ is principal we find that $I$ itself is principal, so $e = u v^3$,
for some $v \in \O_k$ and $u \in \O_k^*$. By rescaling $y$ and $z$ by a factor
$v$ and $v^2$ respectively, we may assume $v=1$, so that $C$ is given by
$x^3+ux^3+u^2z^3=0$. This means that $C$ is isomorphic to one in a fixed
finite set of curves, depending on $k$, namely those curves given by
$x^3+wy^3+w^2z^3=0$ where $w\in \O_k^*$ runs over a set of representatives
of $\O_k^*/(\O_k^*)^3$. For some fields $k$, none of these curves satisfy
both (2) and (3) so that we have a contradiction. For $k=\Q$ for instance,
the group $\O_k^*/(\O_k^*)^3$ is trivial and the curve $x^3+y^3+z^3=0$ does not
satisfy (3) as it contains the point $[0:-1:1]$. For any imaginary quadratic
field the same argument holds, except for $k = \Q(\sqrt{-3})$, where
$\O_k^*/(\O_k^*)^3$ is generated by a primitive cube root $\zeta$ of unity.
In that case there is one extra isomorphism class represented by the curve given by
$x^3+\zeta y^3 +\zeta^2 z^3=0$, which contains the rational point $[1:1:1]$.
We conclude that if $k=\Q$ or $k$ is an imaginary quadratic field whose class
number is not a multiple of $3$, then conditions (2) and (3) of Theorem
\ref{main} imply condition (1).
\end{remark}

\section{Open problems}\label{problems}

\begin{question}
Is there any (not necessarily diagonal) plane cubic curve over a number
field $k$ that has points locally everywhere, but no $k$-cubic points?
\end{question}

\begin{question}
Is there any diagonal plane cubic curve over a number
field $k$ that has points locally everywhere, but no $k$-cubic points?
\end{question}

\begin{question}
Is the Brauer-Manin obstruction the only obstruction to the Hasse
principle on K3 surfaces?
\end{question}

\end{document}